\documentclass[a4paper,12pt]{article}
\usepackage{amsfonts}
\usepackage{amsmath}
\usepackage{amsthm}
\usepackage{graphicx}
\usepackage[export]{adjustbox}
\title{An elementary proof of an isoperimetric inequality for paths with finite $p$-variation.}
\author{George Galvin}
\date{August 2016}

\theoremstyle{definition}
\newtheorem{definition}{Definition}[section]
\newtheorem{theorem}{Theorem}[section]
\newtheorem{lemma}[theorem]{Lemma}

\newcommand{\R}{\mathbb{R}}
\newcommand{\Ptn}{\mathcal{P}}

\begin{document}

\maketitle

\begin{abstract}

In this article we will prove that if the continuous closed curve $\gamma : [0, 1] \rightarrow \R^2$ has finite $p$-variation with $p < 2$, then
\begin{equation*}
(\iint\limits_{\R^2}|\eta(\gamma, (x, y))|^q \,dx \,dy)^{1/q} \le  (\frac{1}{2})^\frac{1}{q}(\zeta(\frac{2}{pq})-1)(||\gamma||_{p, [0, 1]})^{\frac{2}{q}}
\end{equation*}
for all $q \in [1, \frac{2}{p})$, where $\eta(\gamma, (x, y))$ is the winding number of $\gamma$ at $(x, y), \zeta$ is the Reimann zeta function, and
$||\gamma||_{p, [0, 1]}$ is the $p$-variation of $\gamma$ on the interval $[0, 1]$.

Our main contribution is that we have explicitly given a bound by known constants, and we have found this by an elementary proof. We are going to be using a method introduced by L.C. Young 
\cite{young} in 1936.
\end{abstract}

\section{Introduction}
Isoperimetric problems have been studied since the time of Ancient Greece. The simplest isoperimetric problem was solved in the second century BC, when the Greek mathematician Zenodorus proved that a circle has greater area than any polygon with the same perimeter . This was later generalised to produce the classical isoperimetric inequality, which states that for any shape with perimeter  $L$ and area $A$, 
\begin{equation*}
L^2 - 4\pi A \ge 0,
\end{equation*}
with equality when the shape is a circle. \cite{blasjo}

This inequality was generalised by Banchoff and Pohl \cite{banchoff} to include curves that self-intersect. The result of their paper 
was that if $\gamma$ is a two-dimensional closed path with finite length, and $L$ is the length of $\gamma$, and $\eta(\gamma, (x,y))$ is the winding number
(which will be defined below) of $\gamma$ at $(x, y)$, then

\begin{equation*}
L^2 - 4\pi\iint\limits_{\R^2}\eta(\gamma, (x, y))^2 \,dx \,dy \ge 0.
\end{equation*}

More recently it has been shown \cite{recent} that for $1 \le p < 2$, and for all $q < \frac{2}{p}$, there exists $C_{p, q} > 0$ such 
that for all paths $\gamma : [0, 1] \rightarrow \R^2$ with finite $p$-variation, 

\begin{equation}
(\iint\limits_{\R^2}|\eta(\gamma, (x, y))|^q \,dx \,dy)^{1/q} \le  C_{p, q}\max(||\gamma||_{p, [0, 1]}, ||\gamma||^p_{p, [0, 1]}),
\end{equation}
where $||\gamma||_{p, [0, 1]}$ is the $p$-variation (which will also be defined below) of $\gamma$ on the interval $[0, 1]$.

In this article, our main contribution is that we have bounded this integral using known constants, and we have found this result with an elementary proof.

More precisely, we will prove the following:
\begin{theorem} 
If  the continuous closed curve $\gamma : [0, 1] \rightarrow \R^2$ has finite $p$-variation with $p < 2$, then
\[
(\iint\limits_{\R^2}|\eta(\gamma, (x, y))|^q \,dx \,dy)^{1/q} \le 
 (\frac{1}{2})^\frac{1}{q}(\zeta(\frac{2}{pq})-1)(||\gamma||_{p, [0, 1]})^{\frac{2}{q}} \]
for all $q \in [1, \frac{2}{p})$, where $\eta(\gamma, (x, y))$ is the winding number of $\gamma$ at $(x, y), \zeta$ is the Riemann zeta function, and
$||\gamma||_{p, [0, 1]}$ is the $p$-variation (which will be defined below) of $\gamma$ on the interval $[0, 1]. \label{eq:main}$
\end{theorem}

In section 2 (Preliminary), we are going to state and prove some important theorems and lemmas we are going to use to prove Theorem $\ref{eq:main}$.

In section 3 (Main proof), we are going to use this knowledge to complete this proof. We will use Young's method \cite{young} of successively removing partition points.

This idea has also been used in rough path theory. The interested reader is referred to the texts by Friz and Victoir \cite{friz}, Lyons \cite{lyons}, and Lyons, Caruana and L{\'e}vy \cite{notesgiven}. 
\section{Preliminary}

We start by defining the concepts of the \textit{partition} and $p$\textit{-variation}, which we will use heavily throughout this article. Throughout this section we will let $\gamma$ be a continuous closed path in $\R^2$.

\begin{definition}
We say $\Ptn = (t_0, ..., t_r)$ is a partition of $[0, T]$ if $t_0 = 0, t_r = T$ and $t_0 < t_1 < ... < t_{r - 1} < t_r$.
\end{definition}

\begin{definition}Given a partition $\Ptn = (t_0 < t_1 < ... < t_r)$, define 

\begin{equation*}
\gamma^\Ptn_t = \gamma_{t_i} + \frac{t - t_i}{t_{i+1}-t_i}(\gamma_{t_{i+1}}-\gamma_{t_i}).
\end{equation*}

\end{definition}

This means that if we model $\gamma$ as a path, $\gamma^\Ptn$ will be a polygonal interpolation of $\gamma$.

\begin{definition}
Let $p \ge 1$. Let $\gamma : [0, T] \rightarrow \R^2$ be a continuous function. The $p$-variation of $\gamma$ on $[0, T]$ is defined by
\begin{equation*}
||\gamma||_{p, [0, T]} = (\sup \{ \sum\limits_{j = 0}^{r - 1}|\gamma_{t_j} - \gamma_{t_{j+1}}|^p : (t_0, t_1, ..., t_r) \mbox{ is a partition of }[0, T]\})^\frac{1}{p}.
\end{equation*}
Sometimes we will omit writing the interval $[0, T]$ - in this case the interval is taken to be $[0, 1]$.
\end{definition}

\begin{definition}
Suppose that $(x, y) \in \R^2 \backslash \gamma[0, 1]$. 
Then the function $$S^\gamma : t \rightarrow \frac{\gamma_t - (x, y)}{|\gamma_t - (x, y)|}$$ 
maps $[0, 1]$ to the unit circle centred at $0$.

Let $\theta^\gamma : [0, 1] \rightarrow \R$ be a lift of $S^\gamma$ such that $$S^\gamma_t = (\cos \theta^\gamma_t,
\sin \theta^\gamma_t)\ \forall t \in [0, 1].$$
Then the winding number of $\gamma$ at the point $(x, y)$ is defined as $$\eta(\gamma, (x, y)) = \frac{\theta^\gamma(1) - \theta^\gamma(0)}{2\pi}.$$
The function $\eta(\gamma, (x, y))$ is independent of the lift $\theta^\gamma$ (see chapter 3, Lemma 1 and 2, in \cite{windingnumber}).
\end{definition}

To prove Theorem 1.1 we need to also reference an important lemma on winding numbers:

\begin{lemma} (Theorem 7.2 \cite{complexanalysis}) \label{eq:windingadd}
Let $\beta_1 : [0, 1] \rightarrow \R^2$ and $\beta_2 : [0, 1] \rightarrow \R^2$ be continuous paths such that $\beta_1(1) = \beta_1(0)$ and $\beta_2(1) = \beta_2(0).$ Also let
\begin{equation}
\beta_1 \star \beta_2 (t) =
\begin{cases}
\beta_1(2t),& 0 \le t \le \frac{1}{2}, \\
\beta_2(2t - 1) - \beta_2(0) + \beta_1(1),& \frac{1}{2} \le t \le 1.
\end{cases}
\end{equation}
Then,
\begin{equation}
\eta(\beta_1 \star \beta_2, (x, y)) = \eta(\beta_1, (x, y)) + \eta(\beta_2, (x, y)).
\end{equation}

\end{lemma}

We will also use the Minkowski inequality as displayed below.
\begin{lemma}
Let $f$ and $g$ be measurable functions. Then
\begin{equation}
(\iint\limits_{\R^2}|f(x, y) + g(x,y)|^q \,dx \,dy)^{1/q} \le (\iint\limits_{\R^2}|f(x, y)|^q \,dx \,dy)^{1/q} + 
 (\iint\limits_{\R^2}|g(x, y)|^q \,dx \,dy)^{1/q}. \label{eq:minkowski}
\end{equation}
\end{lemma}

\begin{lemma}
Let $\gamma$ be a continuous closed path in $\R^2$ with finite $p$-variation with $p < 2$. Let $\Ptn = (t_1, ..., t_r)$. Then for any $j_1,...,j_n$ and for all 
$n \in \{1, 2, ..., r - 2\} $,

\begin{multline*} 
(\iint\limits_{\R^2} |\eta(\gamma^{\Ptn \backslash \{t_{j_1}, ..., t_{j_n}\}}, (x, y)) - 
\eta(\gamma^{\Ptn \backslash \{t_{j_1}, t_{j_2}, ..., t_{j_{n+1}}\}}, (x, y))|^q \,dx \,dy) \\
\le  \frac{1}{2 \times 2^\frac{2}{p}}||\gamma||^2_{p, [t_{j_{n+1}-1}, t_{j_{n+1}+1}]}.
\end{multline*} \label{eq:trianglelemma}
\end{lemma}

\begin{proof}
 
\begin{figure}
\includegraphics[left]{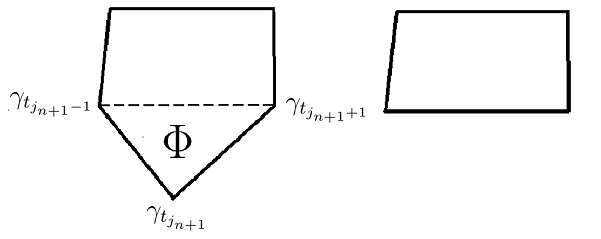}
\caption{
This figure illustrates the removal of a partition point from a pentagonal path.
The diagram on the left represents $\gamma^{\Ptn \backslash \{t_{j_1}, ..., t_{j_n}\}}$, and the diagram on the right represents 
$\gamma^{\Ptn \backslash \{t_{j_1}, t_{j_2}, ..., t_{j_{n+1}}\}}$.}
\end{figure}

Let $\Phi$ denote

\begin{equation}
[\gamma_{t_{j_{n+1}-1}}, \gamma_{t_{j_{n+1}}}] \star [ \gamma_{t_{j_{n+1}}}, \gamma_{t_{j_{n+1}+1}}]
\star [\gamma_{t_{j_{n+1}+1}}, \gamma_{t_{j_{n+1}-1}}],
\end{equation}
with $[a, b]$ denoting the straight line between $a$ and $b$.

By the additivity of the winding numbers, Lemma \ref{eq:windingadd}, the functions
$$(x, y) \rightarrow \eta(\gamma^{\Ptn \backslash \{t_{j_1}, t_{j_2}, ..., t_{j_n}\}}, (x, y))$$
and
$$(x, y) \rightarrow \eta( \gamma^{\Ptn \backslash \{t_{j_1}, t_{j_2}, ..., t_{j_{n+1}}\}}, (x, y))$$
coincide except inside $\Phi$, where the two functions differ by $\eta(\Phi, (x, y))$. Figure 1 is an illustration of this.

Therefore,
\begin{equation}
\eta(\gamma^{\Ptn \backslash \{t_{j_1}, t_{j_2}, ..., t_{j_n}\}}, (x, y)) -  \eta( \gamma^{\Ptn \backslash \{t_{j_1}, t_{j_2}, ..., t_{j_{n+1}}\}}, (x, y)) = \eta(\Phi, (x, y)).
\end{equation}

It follows that
\begin{multline*}
\iint\limits_{\R^2} |\eta(\gamma^{\Ptn \backslash \{t_{j_1}, t_{j_2}, ..., t_{j_n}\}}, (x, y)) -  \eta( \gamma^{\Ptn \backslash \{t_{j_1}, t_{j_2}, ..., t_{j_{n+1}}\}}, (x, y))|^q \,dx \,dy \\
 = \iint\limits_{\R^2} |\eta(\Phi, (x, y))|^q \,dx \,dy.
\end{multline*}

Since $\eta(\Phi, (x, y))$ is $1$ or $-1$ inside $\Phi$ and $0$ outside, 
\begin{equation}
\iint\limits_{\R^2} |\eta(\Phi, (x, y))|^q \,dx \,dy = \mbox{Area}(\triangle_{t_{j_{n+1}-1},
t_{j_{n+1}}, t_{j_{n+1}+1}}),
\end{equation}

where $\triangle_{t_{j_{n+1}-1},
t_{j_{n+1}}, t_{j_{n+1}+1}}$ is the triangle with vertices $\gamma_{t_{j_{n+1}-1}},
\gamma_{t_{j_{n+1}}}, $ and $ \gamma_{t_{j_{n+1}+1}}$.

Now, using the inequality $ab \le \frac{a^2 + b^2}{2}$,
\begin{multline}
\begin{aligned}
\mbox{Area}(\triangle&_{t_{j_{n+1}-1},
t_{j_{n+1}}, t_{j_{n+1}+1}}) \\
&\le \frac{|\gamma_{t_{j_{n+1}-1}} - \gamma_{t_{j_{n+1}}}||\gamma_{t_{j_{n+1}+1}} - \gamma_{t_{j_{n+1}}}|}{2} \\
&= \frac{(|\gamma_{t_{j_{n+1}-1}} - \gamma_{t_{j_{n+1}}}|^{\frac{p}{2}}|\gamma_{t_{j_{n+1}+1}} - \gamma_{t_{j_{n+1}}}|^\frac{p}{2})^\frac{2}{p}}{2} \\
&\le \frac{(|\gamma_{t_{j_{n+1}-1}} - \gamma_{t_{j_{n+1}}}|^p +|\gamma_{t_{j_{n+1}+1}} - \gamma_{t_{j_{n+1}}}|^p )^\frac{2}{p}}{2 \times 2^\frac{2}{p}} \\
&\le \frac{||\gamma||^2_{p, [t_{j_{n+1}-1}, t_{j_{n+1}+1}]}}{2 \times 2^\frac{2}{p}}.
\end{aligned}
\end{multline}

\end{proof}

\section{Main proof}
In this section we will let $\gamma : [0, 1] \rightarrow \R^2$ be a continuous closed path in $\R^2$ with finite $p$-variation, $p < 2$, and let $\Ptn$ be a partition of $[0, 1]$.

We begin by stating an important theorem to be used:

\begin{theorem} (\cite{notesgiven})
For any partition $\Ptn = (t_0 < t_1 < ... < t_k)$ there exists $j \in \{1, ..., k-1\}$ such that
\begin{equation} 
||\gamma||_{p, [t_{j-1}, t_{j+1}]} \le \frac{2}{k-1} ||\gamma||_{p, [0, 1]}. \label{eq:partitionj}
 \end{equation} 
\end{theorem}
We now use Young's method of successively removing points to form an upper bound. Let $\Ptn = (t_0 < ... < t_r)$;
by ($\ref{eq:partitionj}$) with partition $\Ptn$ there exists $j_1 \in \{1, ..., r-1\}$ such that

\begin{equation*} 
||\gamma||^p_{p, [t_{j_1-1}, t_{j_1+1}]} \le \frac{2}{r-1} ||\gamma||^p_{p, [0, 1]}. 
\end{equation*}

We now remove $t_{j_1}$ from $\Ptn$ - therefore we have $r - 2$ points left. By abuse of notation, we relabel these $r - 2$ points as $t_1, ..., t_{r - 2}.$

Now with partition $\Ptn \backslash \{t_{j_1}\}$ there exists $j_2 \in \{1, ..., r - 2\}$ such that 
\begin{equation*} 
||\gamma||^p_{p, [t_{{j_2}-1}, t_{{j_2}+1}]} \le \frac{2}{r - 2} ||\gamma||^p_{p, [0, 1]}.
\end{equation*}

We then generalise this for removing any number of points. For all $n \in \{1, 2,..., r - 2\} $, we can remove $n - 1$ points from $\Ptn$, using abuse of notation to label the $r - n$ points left as $t_1, ..., t_{r - n}$. Then there exists $ j_n \in \{1, ..., r - n\}$ such that 
\begin{equation}
||\gamma||^p_{p, [t_{{j_n}-1}, t_{j_n+1}]} 
\le \frac{2}{r - n} ||\gamma||^p_{p, [0, 1]}. \label{eq:rempoints}
\end{equation}

Removing $r - 2$ points from $\Ptn$, we are left with the path $\gamma^{\Ptn \backslash \{t_{j_1}, t_{j_2}, ..., t_{j_{r - 2}}\}}$. As the partition only has three points left,
and $$\gamma^{\Ptn \backslash \{t_{j_1}, t_{j_2}, ..., t_{j_{r - 2}}\}}(t_0) = \gamma^{\Ptn \backslash \{t_{j_1}, t_{j_2}, ..., t_{j_{r - 2}}\}}(t_r),$$
the path only has two points left. Therefore, 
\begin{equation}
\eta(\gamma^{\Ptn \backslash \{t_{j_1}, t_{j_2}, ..., t_{j_{r - 2}}\}}, (x, y))) = 0\ 
\end{equation}
for all $(x, y)$ outside $\gamma^{\Ptn \backslash \{t_{j_1}, t_{j_2}, ..., t_{j_{r - 2}}\}}$.
 It is then clear that for all $(x, y)$ outside $\gamma^\Ptn$, we have
\begin{multline}
\begin{split}
\eta(\gamma^\Ptn, (x, y)) &= (\eta(\gamma^\Ptn, (x, y)) - \eta(\gamma^{\Ptn \backslash \{t_{j_1}\}}), (x, y)) \\
+ &(\eta(\gamma^{\Ptn \backslash \{t_{j_1}\}}, (x, y)) - \eta(\gamma^{\Ptn \backslash \{t_{j_1}, t_{j_2}\}}, (x, y))) \\
+ &... \\
+ &(\eta(\gamma^{\Ptn \backslash \{t_{j_1}, t_{j_2}, ..., t_{j_{r - 3}}\}}, (x, y)) - 
\eta(\gamma^{\Ptn \backslash \{t_{j_1}, t_{j_2}, ..., t_{j_{r - 2}}\}}, (x, y))). \label{eq:clearsum}
\end{split}
\end{multline}

Therefore, by an extension of the Minkowski inequality ($\ref{eq:minkowski}$) to $r - 2$ functions,
\begin{multline}
\begin{aligned}
&(\iint\limits_{\R^2} |\eta(\gamma^\Ptn, (x, y))|^q \,dx \,dy)^{1/q} \\
&\le (\iint\limits_{\R^2} |\eta(\gamma^\Ptn, (x, y)) - \eta(\gamma^{\Ptn \backslash \{t_{j_1}\}}, (x, y)) |^q \,dx \,dy)^{1/q} \\
&+ (\iint\limits_{\R^2} |\eta(\gamma^{\Ptn \backslash \{t_{j_1}\}}, (x, y)) - 
  \eta(\gamma^{\Ptn \backslash \{t_{j_1}, t_{j_2}\}}, (x, y))|^q \,dx \,dy)^{1/q} \\
&+ ... \\
&+ (\iint\limits_{\R^2} |\eta(\gamma^{\Ptn \backslash \{t_{j_1}, t_{j_2}, ..., t_{j_{r - 3}}\}}, (x, y)) - 
\eta(\gamma^{\Ptn \backslash \{t_{j_1}, t_{j_2}, ..., t_{j_{r - 2}}\}}, (x, y))|^q \,dx \,dy)^{1/q}. \label{eq:minkowskiextension}
\end{aligned}
\end{multline}

We can now create the inequality which forms the base of the final proof.
\begin{lemma} For $q < \frac{2}{p}$,
\begin{equation}
(\iint\limits_{\R^2} |\eta(\gamma^\Ptn, (x, y))|^q \,dx \,dy)^{1/q} \le (\frac{1}{2})^\frac{1}{q}(\zeta(\frac{2}{pq})-1)(||\gamma||_{p, [0, 1]})^{\frac{2}{q}}. \label{eq:trianglebound}
\end{equation}
\end{lemma}

\begin{proof}

From Lemma $\ref{eq:trianglelemma}$,
\begin{multline}
(\iint\limits_{\R^2} |\eta(\gamma^{\Ptn \backslash \{t_{j_1}, t_{j_2}, ..., t_{j_n}\}}, (x, y)) - 
\eta(\gamma^{\Ptn \backslash \{t_{j_1}, t_{j_2}, ..., t_{j_{n+1}}\}}, (x, y))|^q \,dx \,dy)^\frac{p}{2} \\
 \le \frac{1}{2 \times 2^\frac{p}{2}}||\gamma||^p_{p, [t_{j_{n+1}-1}, t_{j_{n+1}+1}]} \\
 \le \frac{1}{2^\frac{p}{2}}(\frac{1}{r - (n + 1)})||\gamma||^p_{p, [0, 1]}\ \mbox{(applying (\ref{eq:rempoints}))} \\
\Rightarrow
(\iint\limits_{\R^2} |\eta(\gamma^{\Ptn \backslash \{t_{j_1}, t_{j_2}, ..., t_{j_n}\}}, (x, y)) - 
\eta(\gamma^{\Ptn \backslash \{t_{j_1}, t_{j_2}, ..., t_{j_{n+1}}\}}, (x, y))|^q \,dx \,dy)^\frac{1}{q} \\
\le \frac{1}{2^\frac{1}{q}}(\frac{1}{(r-(n+1))}||\gamma||^p_{p, [0, 1]})^{\frac{2}{pq}}. \label{eq:deduction}
\end{multline}

We then use this inequality to simplify each integral in ($\ref{eq:minkowskiextension}$),  as follows: 
\begin{multline}
\begin{split}
(\iint\limits_{\R^2} |\eta(\gamma^\Ptn, (x, y))|^q \,dx \,dy)^{1/q}& \\
\le &\frac{1}{2^\frac{1}{q}}(\frac{1}{r-1}||\gamma||^p_{p, [0, 1]})^{\frac{2}{pq}} \\
&+ \frac{1}{2^\frac{1}{q}}(\frac{1}{r-2}||\gamma||^p_{p, [0, 1]})^{\frac{2}{pq}} \\
&+ ... \\
&+ \frac{1}{2^\frac{1}{q}}(\frac{1}{2}||\gamma||^p_{p, [0, 1]})^{\frac{2}{pq}}.
\end{split}
\end{multline}
Then, taking the sum, 
\begin{multline}
\begin{split}
(\iint\limits_{\R^2} |\eta(\gamma^\Ptn, (x, y))|^q \,dx \,dy)^{1/q}& \\
\le &\frac{1}{2^\frac{1}{q}}(\sum\limits_{n=2}^{r-1} (\frac{1}{n})^\frac{2}{pq})
||\gamma||^\frac{2}{q}_{p, [0, 1]} \\
\le &\frac{1}{2^\frac{1}{q}}(\zeta(\frac{2}{pq})-1)||\gamma||^{\frac{2}{q}}_{p, [0, 1]}.
\end{split}
\end{multline}

To bound the above integral from above, we thus require $\zeta(\frac{2}{pq})$ to be finite.
 For this we need $1 < \frac{2}{pq}$ so we need $q < \frac{2}{p}$. 

\end{proof}

Now to complete the proof we create a new partition $\Ptn_n$.

\begin{definition}
Let $\Ptn_n = \{0, \frac{1}{n}, \frac{2}{n}, ..., \frac{n-1}{n}, 1\}$.
\end{definition}

We also need one more lemma; it is known that

\begin{lemma} (\cite{phdthesis}) For all $(x, y)$ outside the image of $\cup^\infty_{n=1} \gamma^{\Ptn_n} \cup \gamma^n$,
\begin{equation}
\liminf_{n\to\infty} \eta(\gamma^{\Ptn_n}, (x, y)) = \eta(\gamma, (x, y)). \label{eq:liminflemma} 
\end{equation}
\end{lemma}

We are now in a position to complete the proof of Theorem 1.1.
\begin{proof}

We can take the limit infinum of both sides of ($\ref{eq:trianglebound}$) as the number of partition points $n$ goes to infinity:

\begin{equation*}
\liminf_{n\to\infty}(\iint\limits_{\R^2} |\eta(\gamma^{\Ptn_n}, (x, y))|^q \,dx \,dy)^{1/q} \le  
\liminf_{n\to\infty}(\frac{1}{2^\frac{1}{q}}(\zeta(\frac{2}{pq})-1)||\gamma||_{p, [0, 1]}^{\frac{2}{q}}).
\end{equation*}

As the right hand side is constant, its limit infinum is itself, so we can rewrite the equation as follows:

\begin{equation}
\liminf_{n\to\infty}(\iint\limits_{\R^2} |\eta(\gamma^{\Ptn_n}, (x, y))|^q \,dx \,dy)^{1/q} \le  
\frac{1}{2^\frac{1}{q}}(\zeta(\frac{2}{pq})-1)||\gamma||_{p, [0, 1]}^{\frac{2}{q}}. \label{eq:liminf}
\end{equation}

Due to Fatou's Lemma,

\begin{equation}
(\iint\limits_{\R^2}\liminf_{n\to\infty} |\eta(\gamma^{\Ptn_n}, (x, y))|^q \,dx \,dy)^{1/q} \le \frac{1}{2^\frac{1}{q}}
(\zeta(\frac{2}{pq})-1)||\gamma||_{p, [0, 1]}^{\frac{2}{q}}. \label{eq:fatouextension}
\end{equation}

Therefore due to ($\ref{eq:liminflemma}$),

\begin{equation*}
(\iint\limits_{\R^2}|\eta(\gamma, (x, y))|^q \,dx \,dy)^{1/q} 
\le  \frac{1}{2^\frac{1}{q}}(\zeta(\frac{2}{pq})-1)||\gamma||_{p, [0, 1]}^{\frac{2}{q}}.
\end{equation*}

\end{proof}

\section{Acknowledgements}
I would like to acknowledge the support of the EPSRC Vacation Bursary. I would like to thank my supervisor Horatio Boedihardjo for his support and insights. I also thank the referee for their useful suggestions and comments.

\newpage

\end{document}